\documentclass[11pt,a4paper]{amsart}

\newcommand{\R}{\mathbb{R}}

\usepackage{amssymb,amsmath,amsthm,enumerate}
\usepackage{graphicx}
\usepackage{color}

\usepackage{hyperref}
\usepackage{cite}

\newtheorem{theorem}{Theorem}[section]
\newtheorem{lemma}[theorem]{Lemma}
\newtheorem{proposition}[theorem]{Proposition}
\newtheorem{corollary}[theorem]{Corollary}
\newtheorem{definition}[theorem]{Definition}

\makeatletter
\@addtoreset{theorem}{section}
\@addtoreset{theorem}{subsection}
\makeatother

\usepackage{url}
\begin{document}

\title{Counting joints in vector spaces over arbitrary fields}
\author{Anthony Carbery and Marina Iliopoulou}
\address{School of Mathematics and Maxwell Institute for Mathematical Sciences, University of Edinburgh, 
Edinburgh, EH9 3JZ, UK}
\address{School of Mathematics, University of Birmingham, Birmingham, B15 2TT, UK}
\email{\href{mailto:A.Carbery@ed.ac.uk}{A.Carbery@ed.ac.uk}, \href{mailto:M.Iliopoulou@bham.ac.uk}
{M.Iliopoulou@bham.ac.uk}}
\maketitle


\section{Introduction}

If $\mathbb{F}$ is an arbitrary field and $n \geq 2$, a point $x \in \mathbb{F}^n$ is {\em a joint formed 
by a finite collection $\mathfrak{L}$ of lines in} $\mathbb{F}^n$ if there exist at least $n$ lines in 
$\mathfrak{L}$ passing through $x$ whose directions span $\mathbb{F}^n$. The problems we consider below are
trivial in the case $n=2$, so from now on we shall assume $n \geq 3$.

The main problem is to bound the number of joints by the correct power of the number of lines forming them.
When $\mathbb{F}= \mathbb{R}$ the result of Quilodr\'{a}n \cite{MR2594983} and Kaplan, Sharir and Shustin \cite{MR2728035}
states that if $\mathfrak{L}$ is a collection of $L$ lines in $\R^n$, 
and $J$ is the set of joints formed by $\mathfrak{L}$, then 

\begin{equation} |J| \leq C_n L^{n/(n-1)} \label{eq:basic} \end{equation} 

where $C_n$ is a constant depending only on the dimension $n$. This is the sharp estimate as is seen by letting
$\mathfrak{L}_j$ be the collection of $M^{n-1}$ lines parallel to the $j$'th standard basis vector $e_j$ passing 
through the points $(k_1, \dots, k_{j-1}, 0, k_{j+1}, \dots , k_n)$ for $k_i \in \{1, \dots , M \}$, and setting 
$\mathfrak{L} = \cup_{j=1}^n \mathfrak{L}_j$.

Prior to these works there were numerous partial results, see 
\cite{Chazelle_Edelsbrunner_Guibas_Pollack_Seidel_Sharir_Snoeyink_1992}, \cite{MR1280600},\cite{MR2047237},\cite{MR2121298}), 
\cite{MR2275834}), \cite{Guth_Katz_2008} and \cite{MR2763049}. 

Kaplan, Sharir and Shustin and Quilodr\'an
exploited properties of polynomials in $\R[x_1, \dots ,x_n]$ which do not hold for polynomials 
with coefficients in general fields. In particular, one can derive less information about a polynomial from its 
gradient in the case of a field of non-zero characteristic than in the case of a field of zero 
characteristic. (For example, in a field of characteristic $p$, the two polynomials $x^p$ and $0$ both 
have zero derivative.) Nevertheless, it has become folklore that the joints estimate \eqref{eq:basic} continues to hold
in arbitrary fields -- see for example the last answer to this 
\href{http://mathoverflow.net/questions/90645/on-the-joints-problem-in-finite-fields/90651#90651}{question} 
in \url{www.mathoverflow.net}. The purpose of this note is to give our argument for this result, which also 
appears in the second author's PhD thesis, \cite{Thesis}.
\footnote{We thank Terry Tao for pointing out, subsequent to the initial posting of 
this note, the articles \cite{Dv} and \cite{MR320022} which also contain treatments of this result.}

\begin{theorem}\label{carberyjoints}  Let $\mathbb{F}$ be any field and $n\geq 3$. 
Let $\mathfrak{L}$ be a finite collection of $L$ lines in $\mathbb{F}^n$, and $J$ the set of joints 
formed by $\mathfrak{L}$. Then

\begin{displaymath} |J| \leq C_n   L^{n/(n-1)},\end{displaymath}

where $C_n$ is a constant depending only on $n$.
\end{theorem}



It is also of interest to count the number of joints of $\mathfrak{L}$ according to multiplicities.
Indeed, for $x$ a joint of $\mathfrak{L}$ let
$$ N(x) = | \{(l_1, \dots , l_n) \in \mathfrak{L}^n \, : \, l_1, \dots , l_n \mbox{  form a joint at  } x\}|.$$
We say that $\mathfrak{L}$ is {\em generic} if whenever $n$ distinct lines of $\mathfrak{L}$ meet, they form 
a joint.  

\begin{theorem}\label{generic}
Let $\mathbb{F}$ be any field and $n\geq 3$. 
Let $\mathfrak{L}$ be a generic finite collection of $L$ lines in $\mathbb{F}^n$, 
and $J$ the set of joints formed by $\mathfrak{L}$.
Then for all $\lambda \geq 1$, 
\begin{displaymath} 
|\{x \in J \, : \, N(x) \geq \lambda \}| 
\leq C_n \frac{ L^{n/(n-1)}}{\lambda^{1/(n-1)}}
\end{displaymath}
where $C_n$ is a constant depending only on $n$.
\end{theorem}

Theorem \ref{generic} follows from Theorem \ref{carberyjoints} by a standard probabilistic 
argument which we briefly  sketch. Given $\mathfrak{L}$ and $\lambda$, we randomly choose 
a subset $\mathfrak{L}'$ of $\mathfrak{L}$, including each line with probability $\lambda^{-1/n}$. 
Then  $|\mathfrak{L}'| 
= \lambda^{-1/n} | \mathfrak{L}|$, and with probability bounded away from zero, if $N(x) > \lambda$ then
$x$ is a joint for $\mathfrak{L}'$. By Theorem \ref{carberyjoints} applied to $\mathfrak{L}'$ we obtain
$$|\{x \in J \, : \, N(x) \geq \lambda \}| \leq C_n \Bigg(\lambda^{-1/n} | \mathfrak{L}|\Bigg)^{n/(n-1)}$$
as required. For full details see \cite{Thesis}. 

The power $1/(n-1)$ of $\lambda$ occuring in Theorem \ref{generic} is optimal as a generic configuration of 
lines all passing through $0$ demonstrates. In the special case of $\mathbb{R}^3$, the stronger estimate
\begin{equation} \label{strong}
\sum_{x \in J} N(x)^{1/2} \leq C L^{3/2}
\end{equation}
has been obtained without the extra hypothesis of genericity, (see
\cite{MI} and \cite{Thesis}). The techniques used for the proof of \eqref{strong} rely on the
topology and continuous nature of euclidean space, as well as on 
algebraic geometric facts that hold only in three dimensions, which 
suggests that their direct application to the setting of different fields 
and to higher dimensions is unlikely. On the other hand, we have been informed by M\'arton Hablicsek \cite{H}
that he has been able to build on recent work of Koll\'ar \cite{Ko} to establish 
$$\sum_{x \in J} N(x)^{1/(n-1)} \leq C_n L^{n/(n-1)}$$
over quite general fields, assuming the genericity hypothesis, thus superseding Theorem \ref{generic}. 
(Koll\'ar's work relies on algebraic geometry and sheaf cohomology and so cannot be considered entirely 
elementary.)

Note the crucial use of the genericity hypothesis in the proof of Theorem \ref{generic}. 
We give an alternative argument for this result below in Section \ref{alt}
which in principle suggests an approach in the non-generic case. In this regard see the remarks in 
Subsection \ref{nongen}.


\textbf{Notation.} In what follows, any expressions of the form $A \lesssim_n B$ mean that $A=O_n(B)$, or, 
in other words, that there exists a non-negative constant $C_n$, depending only on $n$, such that 
$A \leq C_n B$. Similarly, $A \gtrsim_n B$ means that $B
  \lesssim_n A $, while $A \sim B$ means 
that $A \lesssim_n B$ and $A \gtrsim_n B$. We use $\mathbb{Z}_+$ to
denote the set of nonnegative integers.

\section{Preliminaries on polynomials}
Let $\mathbb{F}$ be a field. We emphasise that saying that a polynomial in $\mathbb{F}[x_1, \dots, x_n]$
is non-zero means that it has a non-zero coefficient. While in $\R[x_1, \dots , x_n]$ a non-zero 
polynomial cannot vanish on the whole of $\R^n$, the same does not hold in an arbitrary field setting. 
For example, if $\mathbb{F}$ is a finite field, the non-zero polynomial $x^{|\mathbb{F}|}- x$ in 
$\mathbb{F}[x]$ vanishes on the whole of $\mathbb{F}$.

The basic linear algebra lemma of Dvir \cite{MR2525780} is:


\begin{lemma} \label{dvir} Let $\mathbb{F}$ be any field. For any set $P$ of $m$ points in 
$\mathbb{F}^n$, there exists a non-zero polynomial in $\mathbb{F}[x_1,...,x_n]$, of degree 
$\lesssim_n m^{1/n}$ which vanishes at each point of $P$.
\end{lemma}

\begin{proof} 
We merely notice that if there are fewer equations (one for each point of $P$) than unknowns (the number 
of coefficients of a polynomial of degree at most $d$) then the system of equations

\begin{displaymath}
\sum_{|\alpha| \leq d} c_\alpha x^\alpha = 0, \;\;\;\; x \in P
\end{displaymath}

has a nontrivial solution $\{c_\alpha\}$. Now the number of unknowns is $\sim d^n$, so if $d \sim m^{1/n}$
we can find a polynomial of degree $d$ vanishing on $P$. 
\end{proof}




We now consider the formal (or Hasse) first-order partial derivatives of a polynomial (cf. \cite{MR2648400}).

\begin{definition} Let $\mathbb{F}$ be a field and $f \in \mathbb{F}[x_1,...,x_n]$. For $i \in \{1, \dots, n\}$, 
the \emph{$i$-th formal derivative} $f_{i}$ of $f$ is defined as the coefficient of $z_i$ in $f(x+z)$.
\end{definition}

Clearly when $\mathbb{F} = \R$ the formal derivative and the usual partial derivative agree. It is easy to 
see that the formal derivative is a linear map and that if $m(x_1, \dots ,x_n)=c_{a_1, \dots ,a_n}x_1^{a_1}\cdots x_n^{a_n}$ is a monomial in $\mathbb{F}[x_1, \dots ,x_n]$, then 
\begin{displaymath}
m_{i}(x_1, \dots ,x_n) = a_i \cdot c_{a_1, \dots ,a_n}x_1^{a_1}\cdots  x_i^{a_i-1} \cdots x_n^{a_n}
\end{displaymath}
\begin{displaymath}
\Bigg(:=\bigg(\sum_{k=1}^{a_i}c_{a_1, \dots ,a_n}\bigg)x_1^{a_1}\cdots  x_i^{a_i-1} \cdots x_n^{a_i}\Bigg)
\end{displaymath} 
(which, when $a_i = 0$, means $m_i(x_1, \dots ,x_n) = 0$). Note that, for all $i=1, \dots ,n$, 
$$m_i(x_1,...,x_n)=0\text{ if and only if }a_i \cdot c_{a_1,...,a_n}\;
\bigg(=\sum_{k=1}^{a_i} c_{a_1,...,a_n}\bigg)=0. $$

We define the formal gradient of a polynomial in $\mathbb{F}[x_1, \dots ,x_n]$ as follows:

\begin{definition} Let $\mathbb{F}$ be a field and $f \in \mathbb{F}[x_1, \dots ,x_n]$. The formal gradient of 
$f$ is the element $\nabla{f}$ of $(\mathbb{F}[x_1, \dots ,x_n])^n$, defined as 
\begin{displaymath}\nabla{f}= (f_1, \dots ,f_n).
\end{displaymath}

\end{definition}
When $n=1$ we denote the formal gradient by a prime, as ususal. From now on we refer to the formal derivatives 
and gradients merely as derivatives and gradients.

We would like to be able to derive information about a polynomial $f \in \mathbb{F}[x_1, \dots ,x_n]$ from its gradient. It would be nice to know, for example, that two polynomials in one variable with the same derivative differ by a constant. However, that is not true in general. For example, if $\mathbb{F}$ is a field of characteristic $p$, the derivative of the polynomial $x^{p}$ in $\mathbb{F}[x]$ is equal to $\big(\sum_{k=1}^{p}1\big) x^{p-1}=0$, i.e. it is the zero polynomial. On the other hand, the derivative of the zero polynomial is also the zero polynomial, but $x^p$ and 0 do not differ by a constant as polynomials.

However, the following holds:

\begin{lemma} \label{carbery3}Let $\mathbb{F}$ be a field and suppose $f \in \mathbb{F}[x_1, \dots ,x_n]$ 
satisfies $\nabla{f}=0$.

\emph{(i)} If the characteristic of $\mathbb{F}$ is zero, then $f$ is a constant polynomial.

\emph{(ii)} If the characteristic of $\mathbb{F}$ is $p$ (for some prime $p$), then $f$ is of the form
\begin{displaymath} 
f(x_1, \dots ,x_n)=\sum_{a_i \in \mathbb{Z}_+}
\beta_{a_1, \dots ,a_n} x_1^{p a_1}\cdots x_n^{p a_n}. 
\end{displaymath} 
 \emph{(iii)} If the characteristic of $\mathbb{F}$ is $p$ and $\mathbb{F}$ is algebraically closed, then
$f$ is of the form $f = g^p$ for some $g \in \mathbb{F}[x_1, \dots ,x_n]$.
\end{lemma}

\begin{proof} 
(i) Suppose
\begin{displaymath} 
f(x_1, \dots , x_n)=\sum_{a_i \in \mathbb{Z}_+}
c_{a_1, \dots ,a_n} x_1^{a_1}\cdots x_n^{a_n}.
\end{displaymath}
Since $\nabla{f}=0$, it follows that the polynomial 
\begin{displaymath}
f_i (x_1, \dots ,x_n)=\sum_{a_i \in \mathbb{Z}_+}
a_i \cdot c_{a_1, \dots ,a_n}
x_1^{a_1}\cdots  x_i^{a_i-1} \cdots x_n^{a_n}
\end{displaymath} 
is the zero polynomial, and thus $a_i \cdot c_{a_1, \dots ,a_n}=0$, for all $(a_1, \dots ,a_n)$,
from which we obtain
$ c_{a_1, \dots ,a_n} = 0$ for all $(a_1 , \dots , a_n) \neq (0, \dots, 0)$
since the characteristic of $\mathbb{F}$ is zero. Thus $f$ is a constant 
polynomial.

(ii) Arguing as in (i), with the same notation, we have that $a_i \cdot c_{a_1, \dots ,a_n}=0$ 
for all $(a_1, \dots ,a_n)$. Since the characteristic of $\mathbb{F}$ is $p$, this forces $c_{a_1, \dots ,a_n}$
to be zero unless each $a_i$ is a multiple of $p$.

(iii) Since $\mathbb{F}$ is algebraically closed, with $\beta_{a_1, \dots ,a_n}$ as in (ii), there exist
$\gamma_{a_1, \dots ,a_n} \in \mathbb{F}$ such that 
$$\gamma_{a_1, \dots ,a_n}^p = \beta_{a_1, \dots ,a_n}.$$ 
Let
$$g(x) = \sum_{a_i \in \mathbb{Z}_+} \gamma_{a_1, \dots ,a_n} x_1^{a_1} \dots x_n^{a_n}.$$
Then, expanding binomially and using the fact that the characteristic of $\mathbb{F}$ is $p$, we have
$$g(x)^p = \Bigg(\sum_{a_i \in \mathbb{Z}_+}
 \gamma_{a_1, \dots ,a_n} x_1^{a_1} \dots x_n^{a_n} \Bigg)^p
= \sum_{a_i \in \mathbb{Z}_+} \gamma_{a_1, \dots ,a_n}^p x_1^{p a_1} \dots x_n^{p a_n}$$
$$ = \sum_{a_i \in \mathbb{Z}_+} \beta_{a_1, \dots ,a_n} x_1^{p a_1} \dots x_n^{p a_n} = f(x).$$

\end{proof}

For $x=(x_1, \dots ,x_n)$ and $y=(y_1, \dots ,y_n)$ in $\mathbb{F}^n$, denote by $ \langle x, y \rangle $ 
the element $x_1y_1+ \cdots +x_ny_n$ of $\mathbb{F}$. 

\begin{lemma}\label{nabla} Let $\mathbb{F}$ be a field and $f \in \mathbb{F}[x_1,...,x_n]$. 
Let $l$ be the line $\{v+tb:t \in \mathbb{F}\}$ for some $v \in \mathbb{F}^n$ with direction 
$b \in \mathbb{F}^n \setminus \{0 \}$. If $f|_l (t):=f(v+tb)\in \mathbb{F}[t]$ 
is the restriction of $f$ to $l$, we have 
\begin{displaymath} (f|_l)'(t)= \langle b ,\nabla{f}(v+tb) \rangle.
\end{displaymath}
\end{lemma}

We leave the proof as an easy exercise.

\section{Proof of Theorem \ref{carberyjoints}}

We are now ready to prove Theorem \ref{carberyjoints}. Following \cite{MR2594983}, 
the main tool is:

\begin{proposition}\label{main}
Let $\mathbb{F}$ be a field, let $\mathfrak{L}$ be a finite set of lines in $\mathbb{F}^n$ and suppose 
$K$ is some subset of the set of joints of $\mathfrak{L}$. Suppose that for each line $l \in \mathfrak{L}$ 
we have 
$$ | l \cap K | \geq m.$$
Then 
$$ |K| \gtrsim_n m^n.$$
\end{proposition}

\begin{proof}
In the first place we may assume without loss of generality that $\mathbb{F}$ is algebraically closed.
For if $\mathfrak{L}$ is a finite set of lines in $\mathbb{F}^n$ we can form, in the obvious way, the set of lines 
$\overline{\mathfrak{L}}$ in $\overline{\mathbb{F}^n}$ where $\overline{\mathbb{F}}$ is the algebraic closure
of $\mathbb{F}$. We need to check that every joint of $\mathfrak{L}$ in $\mathbb{F}^n$ is also a joint of
$\overline{\mathfrak{L}}$. Indeed, we notice that if a set of $n$ vectors is linearly independent in
$\mathbb{F}^n$, it remains linearly independent in $\mathbb{E}^n$ where $\mathbb{E}$ is any extension field
of $\mathbb{F}$. (This is because linear independence of a set of $n$
vectors in $\mathbb{F}^n$ is characterised by nonvanishing of the
determinant of the matrix whose columns
are these vectors; this property remains unchanged upon passing to extensions.)


So from now on we assume that $\mathbb{F}$ is algebraically closed.

Suppose for a contradiction that we have $ |K| < C_n m^n$ for a suitable $C_n$. Then there will, by Lemma 
\ref{dvir}, be a non-zero $f \in \mathbb{F}[x_1, \dots, x_n]$ of minimal degree {\em strictly smaller than} $m$, 
which vanishes on $K$.  

Take $l \in \mathfrak{L}$ and consider $f|_l \in \mathbb{F}[t]$. This
has degree strictly smaller than $m$ 
but vanishes at at least $m$ points. Hence it is the zero polynomial, and so is its derivative, which, by Lemma
\ref{nabla}, is $\langle \omega_l ,\nabla{f}(v+t \omega_l) \rangle$ where $\omega_l$ is the direction of $l$.
So $\langle \omega_l ,\nabla{f}(x) \rangle = 0$ for all $x \in l$.

If $x \in K$, there exist $l_1, \dots , l_n \in \mathfrak{L}$ with $x \in l_j$ and 
$\omega_{l_1}, \dots, \omega_{l_n}$ spanning $\mathbb{F}^n$. Hence
$$ \nabla{f}(x) = 0  \; \; \mbox{  for all   } \; \; x \in K. $$

So each component of $\nabla{f}$ vanishes on $K$, and since by definition $f$ was the non-zero polynomial of 
{\em smallest} degree vanishing on $K$ we must have that each component of $\nabla{f}$ is the zero polynomial. 
Hence
$$ \nabla{f} = 0.$$ 

We now use Lemma \ref{carbery3}. If the characteristic of $\mathbb{F}$ is zero, $f$ is a constant. However, 
$f$ vanishes on $K$, and thus $f$ is the zero polynomial which is a contradiction. If the characteristic of 
$\mathbb{F}$ is $p \geq 2$, then $f = g^p$ for some $g \in \mathbb{F}[x_1, \dots, x_n]$ which shares the 
same zero set as $f$, and in particular vanishes on $K$. Now unless $f$ is constant, the degree of $g$ 
will be strictly smaller than the degree of $f$, contradicting the definition of $f$. So $f$ is constant, hence 
zero as it vanishes on $K$, which is once again a contradiction.
\end{proof}

The contrapositive of Proposition \ref{main} is:

\begin{corollary}\label{cor}
Suppose that $\mathfrak{L}$ is a finite collection of lines in $\mathbb{F}^n$ and that $K$ is a subset of the set 
of joints of $\mathfrak{L}$. Then there exists a line $l \in \mathfrak{L}$
such that 
$$ |l \cap K|  \lesssim_n |K|^{1/n}.$$
\end{corollary}

The remainder of the proof of Theorem \ref{carberyjoints} is the same as in the solution of the joints problem 
in $\R^n$ by Quilodr\'an (see \cite{MR2594983}).

\begin{corollary}\label{cor2}
Suppose that $\mathfrak{L}$ is a finite collection of $L$ lines in $\mathbb{F}^n$ and that $K$ is a subset 
of the set of joints of $\mathfrak{L}$. Then $K$ can be partitioned into at most $L$ sets, each of cardinality
at most $C_n |K|^{1/n}$.
\end{corollary}

\begin{proof}
Set $K_0 = K$. By Corollary \ref{cor}, there is a line $l_1 \in \mathfrak{L}$ with at 
most $C_n |K|^{1/n}$ members of $K$ on it. Let $\mathfrak{L}_1 = \mathfrak{L} \setminus \{l_1 \}$ 
let $K_1 = K \cap l_1^c$. Note that $K_1$ is a subset of the joints of $\mathfrak{L}_1$.

By  Corollary \ref{cor} once more, there is a line $l_2 \in \mathfrak{L}_1$ 
with at most $C_n |K_1|^{1/n} \leq C_n |K|^{1/n} $ members of $K_1$ on it.
Let $\mathfrak{L}_2 = \mathfrak{L}_1 \setminus \{l_2\}$ and let $K_2 = K_1 \cap l_2^c$. 
Note that $K_2$ is a subset of the joints of $\mathfrak{L}_2$.

We continue in this way, and once there are fewer than $n$ lines remaining, there are no joints remaining, 
and the process stops. This happens in at most $L$ steps. Now $K$ is the disjoint union of the at most 
$L$ sets $l_j \cap K_{j-1}$, each of which has cardinality at most $C_n |K|^{1/n}$.

\end{proof}

The proof of Theorem \ref{carberyjoints} is now immediate as Corollary \ref{cor2} shows that
\begin{displaymath}|J| \lesssim_n L \cdot |J|^{1/n},
\end{displaymath}
which gives
\begin{displaymath}|J| \lesssim_n L^{\frac{n}{n-1}}
\end{displaymath}
upon rearranging.


\section{An alternative argument for Theorem \ref{generic}}\label{alt}


While the proof sketched above for Theorem \ref{generic} is very straightforward, the alternative argument which follows 
is perhaps more instructive. To set some notation, if $\mathfrak{L}$ is a 
set of lines in $\mathbb{F}^n$ let 
$$ N_\mathfrak{L}(x) 
= | \{(l_1, \dots , l_n) \in \mathfrak{L}^n \, : \, l_1, \dots , l_n \mbox{  form a joint at  } x\}|.$$
and for $\lambda \geq 1 $ let 
$$J_\lambda(\mathfrak{L}) = \{ x \in \mathbb{F}^n \; : \; N_\mathfrak{L}(x) \geq \lambda \} \mbox{  and  }
J(\mathfrak{L}) = J_1 (\mathfrak{L}).$$

\begin{proposition}\label{choosing}
Let $\mathfrak{L}$ be a generic finite collection of lines in $\mathbb{F}^n$. Then for each 
$x \in J_\lambda(\mathfrak{L})$ we can choose $\sim \lambda^{1/n}$ lines from $\mathfrak{L}$, each 
containing $x$, such that for each line $l$, the number of $x \in J_\lambda(\mathfrak{L})$ choosing 
$l$ is $\lesssim_n |J_\lambda(\mathfrak{L})|^{1/n}$.
\end{proposition}

\begin{proof}
We may assume that $\lambda \gg  n^n$.

By Corollary \ref{cor}, there exists an $l_1 \in \mathfrak{L}$ such that
$$| l_1 \cap J_\lambda (\mathfrak{L})| \lesssim_n |J_\lambda(\mathfrak{L})|^{1/n}.$$
Let $\mathfrak{L}_1 = \mathfrak{L} \setminus\{l_1\}$. Then there exists an $l_2 \in \mathfrak{L}_1$ such that
$$| l_2 \cap J_\lambda (\mathfrak{L}) \cap J(\mathfrak{L}_1)| 
\lesssim_n |J_\lambda(\mathfrak{L}) \cap J(\mathfrak{L}_1)|^{1/n} \leq |J_\lambda(\mathfrak{L})|^{1/n}.$$
Let $\mathfrak{L}_2 = \mathfrak{L}_1 \setminus\{l_2\}$ and continue in this way to obtain $l_m \in \mathfrak{L}_{m-1}$
such that
$$| l_m \cap J_\lambda (\mathfrak{L}) \cap J(\mathfrak{L}_{m-1})| 
\lesssim_n |J_\lambda(\mathfrak{L}) \cap J(\mathfrak{L}_{m-1})|^{1/n} \leq |J_\lambda(\mathfrak{L})|^{1/n},$$
and then define $\mathfrak{L}_m = \mathfrak{L}_{m-1} \setminus\{l_m\}$.
(The process stops when we have arrived at some last nonempty 
$J(\mathfrak{L}_{m_\ast -1})$ and chosen some last $l_{m_\ast}$ such
that $J(\mathfrak{L}_{m_\ast}) = \emptyset$, and in any case before we reach $|\mathfrak{L}| - n + 1$ 
steps.)

We say that $x \in J_\lambda(\mathfrak{L})$ chooses $l_m$ iff $ x \in l_m \cap J_\lambda (\mathfrak{L}) 
\cap J(\mathfrak{L}_{m-1})$. By construction the number of $x$ choosing $l_m$ 
is $\lesssim_n |J_\lambda(\mathfrak{L})|^{1/n}$, and every member of
$J_\lambda(\mathfrak{L})$ chooses some $l_m$. (Indeed, if
  $x \in J_\lambda (\mathfrak{L})$ does not choose $l_1$, we have $x
  \notin l_1$ and hence $x \in J(\mathfrak{L}_1)$. 
If now $x$ does not choose $l_2$, we have $x \notin l_2$ and hence $x \in J(\mathfrak{L}_2)$. 
Continuing in this way we get that if $x \in J_\lambda (\mathfrak{L})$ does not choose $l_m$, for all $m \in \{1,...,m_\ast\}$ then $x
\notin l_1 \cup \dots \cup l_{m_\ast}$. So $x \in J(\mathfrak{L}_{m_{\ast}})$, which is a contradiction to the emptiness
of this set.)

We still need to show that each $x \in J_\lambda(\mathfrak{L})$ chooses $\gtrsim_n \lambda^{1/n}$ lines $l_m$. 
Suppose that $x$ chooses $l_{m_1}, \dots ,l_{m_r}$ (with $ r \geq 1$ and $m_1 < m_2 < \dots < m_r$), and 
no other lines $l_m$. Then $x$ is {\em not} a joint of $\mathfrak{L}_{m_r} =
\{l_{m_r +1}, l_{m_r + 2}, \dots \}$. 
(Indeed, if $x \in J(\mathfrak{L}_{m_r})$, then, since it
  does not choose $l_{m_r +1}$, it follows that $x \notin l_{m_r +1}$ and
  hence $x \in J(\mathfrak{L}_{m_r + 1})$. Since $x$
does not choose $l_{m_r +2}$, we have that $x \notin l_{m_r +2}$ and
  hence $x \in J(\mathfrak{L}_{m_r + 2})$. Continuing in this way
  we again contradict the emptiness of $J(\mathfrak{L}_{m_{\ast}})$.)
So $x$ belongs to at most $n-1$ members of $\mathfrak{L}_{m_r}$, by the genericity hypothesis. 
Since $x$ belongs to $\gtrsim_n \lambda^{1/n}$ members 
of $\mathfrak{L}$, it must be that it belongs to $\gtrsim_n
\lambda^{1/n} - n \sim \lambda^{1/n}$ 
members of 
 $\mathfrak{L}\setminus \mathfrak{L}_{m_r} = \{l_1, \dots, l_{m_r}\}$. 
Hence $r \gtrsim_n \lambda^{1/n}$. 

\end{proof}

With Proposition \ref{choosing} in hand, the proof of Theorem \ref{generic} is immediate: let $c(x,l) = 1$ if $x$ chooses $l$ and 
$c(x,l) = 0$ otherwise. Then 
$$ \sum_{l\in \mathfrak{L}} \sum_{x \in J_{\lambda} (\mathfrak{L}) } c(x,l) \lesssim_n |\mathfrak{L}| |J_\lambda(\mathfrak{L})|^{1/n}$$
and 
$$ \sum_{x\in J_{\lambda} (\mathfrak{L})} \sum_{l \in \mathfrak{L}} c(x,l) \gtrsim_n |J_\lambda(\mathfrak{L})| \lambda^{1/n}.$$
So  
$$|J_\lambda(\mathfrak{L})| \lambda^{1/n} \lesssim_n |\mathfrak{L}| |J_\lambda(\mathfrak{L})|^{1/n},$$
which upon rearrangement gives 
$$|J_\lambda(\mathfrak{L})| \lesssim_n |\mathfrak{L}|^{n/(n-1)}/\lambda^{1/(n-1)}$$
as required.

\subsection{Remarks on the genericity hypothesis}\label{nongen}

In the rest of this note we discuss the hypothesis of genericity in the statements of Theorem \ref{generic} and 
Proposition \ref{choosing}. It is certainly needed for the proofs we have given, but one might hope that the results 
remain true without it.

To see that genericity is needed for the algorithm in the proof of Proposition \ref{choosing} to work,  
let us consider $\mathbb{R}^3$, and let $\mathfrak{L}$ consist of the $x_3$-axis together 
with $M \gg 1 $ lines with distinct directions 
in the plane $x_3 = 0$. Then $\mathfrak{L}$ is non-generic,
$0$ is the only joint of $\mathfrak{L}$ and $N(0) \sim M^2$. Set $\lambda = M^2$.   
We can choose $l_1$ to be the $x_3$-axis; then $J(\mathfrak{L}_1) = \emptyset$ and no further lines are chosen. So
$0$ is not in $\gtrsim \lambda^{1/3} = M^{2/3}$ lines chosen by the procedure. On the other hand, if we avoid choosing 
the $x_3$-axis at any step, the algorithm does work. 

One may wonder, therefore, if, for any configuration of lines and joints, there is always an appropriate 
{\em choice} of line at each step of the algorithm, which ultimately implies the conclusion of Proposition \ref{choosing}. 
The answer is no: in other words, Proposition \ref{choosing}, as stated (i.e. for \textit{all} joints in 
$J_{\lambda}(\mathfrak{L})$, with the notation of the Proposition), cannot in general be deduced by an appropriate 
application of our algorithm in the non-generic case. This is demonstrated by the following example.

Let $\mathbb{F}$ be a finite field of cardinality $p$. In $\mathbb{F}^3$, we consider the set $\mathfrak{L}$ of all lines 
in the horizontal plane $\{x_3=0\}$, together with a vertical line through each point of that plane. The set 
$J$ of joints formed by $\mathfrak{L}$ is the plane $\{x_3=0\}$ (so $|J|=p^2$), while each joint in $J$ has 
multiplicity $\sim p^2$. The collection $\mathfrak{L}$ is non-generic.

Now, the first step of our algorithm requires the removal of a line in $\mathfrak{L}$ containing $\lesssim |J|^{1/3}$, 
i.e. $\lesssim p^{2/3}$, joints. Since each horizontal line of $\mathfrak{L}$ contains $p$ joints, that line has to 
be a vertical one, after the removal of which the joint $x$ it contains is not a joint any more for the remaining 
collection of lines. Therefore, $x$ will only choose one line of $\mathfrak{L}$ via our algorithm, and not 
$\gtrsim N(x)^{1/3}\sim p^{2/3}$, which is what the conclusion of Proposition \ref{choosing} would require.

Nevertheless, the conclusion of Proposition \ref{choosing} {\em does} hold for 
this example. Indeed, for each 
$x \in J$ (each of which has multiplicity $\sim p^2$), it is possible to choose $\sim p^{2/3}$ lines from 
$\mathfrak{L}$, each containing $x$, such that, for each $l \in \mathfrak{L}$, the number of 
$x \in J$ choosing $l$ is $\lesssim |J|^{1/3} \sim p^{2/3}$.

To see this, let $1 \ll k \ll p$, and partition $\mathbb{F}$ into sets $S_1$, \dots , $S_k$, each with
 $\sim p/k = m$ members. For $j=1, \dots ,k$, let each point $(x_0,y_0)\in \mathbb{F}\times S_j$ choose 
all the lines through it with ``slopes" in $S_j$ (i.e. all the lines of the form $\{(x,y) \, : \, y-y_0=b(x-x_0)\}$, 
where $b \in S_j$). Thus, each point is choosing $m$ lines through it. Now, each 
line $l \in \mathfrak{L}$ must have slope in $S_j$ for some $j \in \{1, \dots ,k\}$ and thus the points $(x,y)$ 
choosing it are the ones that are on the line and have $y \in S_j$. There are clearly $\sim m$ such points, 
except when the slope of the line is $0$, in which case up to $p$ joints may be choosing it.

This means that, for $1 \ll k \ll p$, each point $x \in J$ can choose $\sim p/k$ lines through 
it from $\mathfrak{L}$ (the ones described above, except the ones of slope 0), such that, for each 
$l \in \mathfrak{L}$, the number of $x \in J$ choosing $l$ is $\sim p/k$. By setting $k\sim p^{2/3}$, 
it follows that the conclusion of Proposition \ref{choosing} holds for this example. 

It is possible to consider alternative, weaker versions of Proposition \ref{choosing} (in the non-generic setting) 
which would have the same consequences for Theorem \ref{generic}. Indeed, it is clear that Theorem \ref{generic} in the 
non-generic case could be deduced from a non-generic version of Proposition \ref{choosing} in which the conclusion holds 
not necessarily for \textit{all} joints in $J_{\lambda}(\mathfrak{L})$ (with the notation of the Proposition), 
but for \textit{a large proportion} of $J_{\lambda}(\mathfrak{L})$, i.e. 
for $\gtrsim |J_{\lambda}(\mathfrak{L})|$ elements of $J_{\lambda}(\mathfrak{L})$. 



\addcontentsline{toc}{chapter}{Bibliography}

\end{document}